\documentclass[12pt]{article}

\usepackage[
a4paper,
textwidth=16cm,
textheight=22cm,
hcentering,
vcentering
]{geometry}

\usepackage{amsmath, amssymb, amsfonts, amsthm}
\usepackage{mathrsfs}
\usepackage{graphicx}
\usepackage{caption}
\usepackage{float}
\usepackage{hyperref}

\newtheorem{lem}{Lemma}[section]
\newtheorem{thm}[lem]{Theorem}
\newtheorem{cor}[lem]{Corollary}

	\title{A characterization of $4$-connected graphs with no 6-wheel minor}

\begin{document}
	\author{Zijun Chen, Yuqi Xu,\enspace Weihua Yang\footnote{Corresponding author. E-mail: ywh222@163.com,~yangweihua@tyut.edu.cn}\\ 	\\
	\small Department of Mathematics, Taiyuan University of Technology,\\
	\small Taiyuan Shanxi-030024, China\\}
	\date{}

\maketitle

	\begin{abstract}	
	For each integer $n\geq 3$, let $W_n$ denote the wheel graph obtained by connecting a single vertex to all vertices of a cycle of length $n$. In particular, $W_6$ is obtained from the Petersen graph by contracting three edges incident with a common vertex. In this paper, we determine all $4$-connected graphs that do not contain $W_6$ as a minor.
	\end{abstract}
	
\vspace{0.5cm}


\section{Introduction}

	All graphs considered in this paper are simple. Let $G$ and $H$ be graphs. We say that $H$ is a \emph{minor} of $G$, denoted by $H\preceq G$, if $H$ can be obtained from $G$ by a sequence of edge deletions and edge contractions. A graph $G$ is said to be \emph{$H$-minor-free} if no minor of $G$ is isomorphic to $H$. Many important problems in graph theory can be formulated in terms of $H$-minor-free graphs. In one of the most famous conjectures related to the Petersen graph, Tutte~\cite{tutte1966algebraic} proposed that every bridgeless graph without a Petersen minor admits a nowhere-zero 4-flow.

	Although no explicit structure theorem is currently known for Petersen-minor-free graphs, several results are known for graph classes excluding smaller minors. Among $3$-connected graphs with twelve edges, there are three classical results: Maharry~\cite{maharry2000characterization} characterized cube-minor-free graphs, Ding~\cite{ding2013characterization} characterized octahedron-minor-free graphs, and Maharry and Robertson~\cite{maharry2016structure} characterized $V_8$-minor-free graphs. For each integer $n\geq 3$, let $W_n$ denote the \emph{wheel graph} obtained by connecting a single vertex to all vertices of a cycle $C_n$. Dirac~\cite{dirac1952property} determined that $W_3$-minor-free graphs are the series-parallel graphs. Tutte~\cite{tutte1961theory} showed that $W_3$ is the only $3$-connected $W_4$-minor-free graph. Oxley~\cite{oxley1989regular} characterized $3$-connected $W_5$-minor-free graphs. Gubser~\cite{gubser1993planar} characterized the planar $3$-connected $W_6$-minor-free graphs. Gaslowitz~\cite{gaslowitz2018characterizations} characterized the planar $4$-connected $DW_6$-minor-free graphs. Note that $W_6$ has twelve edges and can be obtained from the Petersen graph by contracting three edges incident with a common vertex. In this paper, we characterize all $4$-connected $W_6$-minor-free graphs.

	Let $G\setminus e$ denote the graph obtained from $G$ by deleting an edge $e$. We use $G/e$ to denote the graph obtained from $G$ by first contracting an edge $e$ and then deleting all but one edge from each parallel family. For each integer $n\geq 3$, let $DW_n$ denote a \emph{double-wheel}, a graph on $n+2$ vertices obtained from a cycle $C_n$ by adding two nonadjacent vertices $u$ and $v$ and connecting each of them to all vertices of the cycle. Moreover, let $DW^+_n$ be the graph obtained from $DW_n$ by joining $u$ and $v$. For each integer $n\geq 5$, let $C^2_n$ be the graph obtained from a cycle $C_n$ by adding edges between all pairs of vertices at distance two on the cycle. Let $K_{4,3}$ be the complete bipartite graph with bipartition $(V_1,V_2)$, where $|V_1|=4$ and $|V_2|=3$. Let $K^{i,j}_{4,3}$, with $0\leq i \leq 6$ and $0\leq j\leq 6$, denote the graph obtained by adding $i$ edges between vertices of $V_1$ and $j$ edges between vertices of $V_2$. Our main result is as follows.

 	\begin{thm} \label{thm1.1}
 	 A $4$-connected graph $G$ is $W_6$-minor-free if and only if $G$ belongs to $\{C^2_n:5\leq n\leq 8\}$ $\cup$ $\{DW^+_4, K_6\setminus e, K_6, C^2_7+e, DW_5, K^{3,0}_{4,3}, K^{3,1}_{4,3}, K^{4,0}_{4,3}, \Gamma_1, K^{4,1}_{4,3}\}$, where the graphs $K^{3,0}_{4,3}, K^{3,1}_{4,3}, K^{4,0}_{4,3}, \Gamma_1$, and $K^{4,1}_{4,3}$ are shown in Figure~\ref{fig1}.
 	\end{thm}
 
	\begin{figure}[H]
	\input{result.tpx}
	\label{fig1}
	\end{figure}

	The following is an immediate consequence of Theorem~\ref{thm1.1}.
	
\begin{cor}
	Every $4$-connected $W_6$-minor-free graph $G$ is one of the following: a graph on at most six vertices, a graph on seven vertices with maximum degree at most five, or a graph isomorphic to $C^2_8$.
\end{cor}

\section{Preliminaries}

	Our main tool is a chain theorem for $4$-connected graphs due to Qin and Ding \cite{qin2019chain}. To explain this result, we need a few definitions.
	Let $v$ be a vertex of a $4$-connected graph $G$. Let $N_G(v)$ denote the set of vertices adjacent to $v$ in $G$. A \emph{split} of $v$ in a graph $G$ produces a new graph $G'$ as follows. Let $(X,Y)$ be a partition of $N_G(v)$ such that $|X|,|Y|\geq 2$. Splitting $v$ produces a graph $G'$ obtained from $G \setminus v$ by adding two new adjacent vertices $x$ and $y$, joining $x$ to all vertices in $X$, and joining $y$ to all vertices in $Y$. The graph $G'$ is called a split of $G$. It is worth noting that the graph $G'$ remains $4$-connected \cite{ding2016graphs}. The \emph{degree} of a vertex $v$, denoted by $d(v)$, is the number of edges incident with $v$. A graph is \emph{cubic} if every vertex has degree 3. A cubic graph $G$ is called \emph{cyclically $4$-connected} if no set of fewer than four edges separates $G$ into two components each containing a cycle. The \emph{line graph} $L(G)$ of $G$ is the graph whose vertices correspond to the edges of $G$, with two vertices adjacent if and only if the corresponding edges in $G$ share a common end vertex. Let $\mathcal{L}$ denote the class of line graphs of cyclically $4$-connected cubic graphs. Let $\mathcal{C}=\{C^2_{n}: n\geq 5\}$. Notice that $C^2_{n}$ is $4$-connected and is planar if and only if $n$ is even. Let $G+e$ denote the graph obtained from a graph $G$ by adding an edge $e$ between two nonadjacent vertices. 
  
	Martinov~\cite{martinov1982uncontractable} proved the following classical result: for every $4$-connected graph $G$, there exists a sequence of $4$-connected graphs $G_0, G_1, \ldots, G_n$, called a $(G, G_n)$-\emph{chain}, such that $G_0=G$, $G_n \in \mathcal{C} \cup \mathcal{L}$, and every $G_i$ $(i<n)$ has an edge $e_i$ for which $G_i/e_i = G_{i+1}$. The following strengthened result is an important tool to characterize all $4$-connected $W_6$-minor-free graphs.

	\begin{thm}[\cite{qin2019chain}]\label{thm2.1}
	Let $G$ be a $4$-connected graph not in $\mathcal{C} \cup \mathcal{L}$. If $G$ is planar, then there exists a $(G,C^2_6)$-chain; if $G$ is nonplanar, then there exists a $(G,C^2_5)$-chain.
	\end{thm}

	Our proof of Theorem~\ref{thm1.1} is divided into two parts. Graphs in $\mathcal{C}\cup\mathcal{L}$ are considered in Section 3, while extensions from $C^2_5$ and $C^2_6$ are discussed in Section 4.

\section{The case $G\in\mathcal{C}\cup\mathcal{L}$}

	In this section we consider graphs belonging to $\mathcal{C}\cup\mathcal{L}$. \emph{Adding a handle} to a graph involves the following operation: for two nonadjacent edges $e_1$ and $e_2$ in $G$, subdivide both edges and then add a new edge connecting the internal vertices of the two paths. A graph obtained from a graph $J$ by repeatedly subdividing edges is called a \emph{topological} $J$. A graph $H$ is \emph{topologically contained} in a graph $G$ if there exists a subgraph of $G$ that is isomorphic to a topological $H$.

	\begin{lem}[\cite{wormald1979classifying}]\label{Lemma3.1}
 	The class of all cyclically $4$-connected cubic graphs can be generated from $K_{3,3}$ and the $\mathrm{cube}$ by repeatedly adding handles.
	\end{lem}

	\begin{lem}[\cite{maharry1999excluded}]\label{Lemma3.2}
 	If $H$ is topologically contained in $G$, then $L(H)\preceq L(G)$.
	\end{lem}
 
	\begin{lem}\label{Lemma3.3}
	Let $G\in \mathcal{L}$ be a $4$-connected graph. Then $G$ contains $W_6$ as a minor.
	\end{lem}
	
	\begin{proof}
	Let $G=L(H)$, where $H$ is a cubic cyclically $4$-connected graph. By Lemma~\ref{Lemma3.1}, $H$ contains a topological $K_{3,3}$ or a topological $\mathrm{cube}$. Thus, either $L(K_{3,3})\preceq L(H)$ or $L(\mathrm{cube})\preceq L(H)$ by Lemma~\ref{Lemma3.2}. A simple check shows that both $L(K_{3,3})$ and $L(\mathrm{cube})$ contain a $W_6$-minor. As illustrated in Figure~\ref{fig2}, heavy lines indicate contracted edges, and dashed lines indicate deleted edges. Thus, $W_6\preceq G$.
	\end{proof}
 
	\begin{figure}[H]
	\input{L_K3,3.tpX}
	\label{fig2}
	\end{figure}

	\begin{lem}\label{Lemma3.4}
	Let $G\in \mathcal{C}$ be a $4$-connected graph. Then $G$ is $W_6$-minor-free if and only if $G$ belongs to $\{C^2_n:5\leq n\leq 8\}$.
	\end{lem}

	\begin{proof}
	
	\noindent\textbf{Case 1.} $5\leq n\leq 8$.
	Since $W_6$ has seven vertices, whereas $C^2_5$ and $C^2_6$ each have fewer vertices, it follows that $W_6$ cannot be a minor of either. If $W_6\preceq C^2_7$, then $W_6$ can be obtained from $C^2_7$ by edge deletions. However, the maximum degree of $C^2_7$ is 4, while that of $W_6$ is 6, leading to a contradiction. If $W_6 \preceq C^2_8$, then exactly one edge must be contracted. Label the ordered vertices of $C^2_8$ as $\{v_1,v_2,\ldots,v_8\}$. Up to symmetry, there are only two possible edge contractions, namely contracting $v_1v_2$ or $v_1v_3$; the resulting graphs have maximum degrees 4 and 5, respectively. Since both values are less than 6, we obtain a contradiction.

	\begin{figure}[H]
	\input{c210.tpx}
	\end{figure}

	\noindent\textbf{Case 2.} $n\geq 9$. 

	For every $n\geq 11$, the graph $C_n^2$ contains either $C_9^2$ or $C_{10}^2$ as a minor. Since both $C_9^2$ and $C_{10}^2$ contain a $W_6$-minor, it follows that $W_6 \preceq C_n^2$ for every $n\geq 9$.
	\end{proof}

\section{Extensions of $C^2_5$ and $C^2_6$}

	We first record a simple criterion for detecting a $W_6$-minor in a graph on seven vertices. If a graph $G$ on seven vertices has a vertex $v$ of degree 6 and $G\setminus v$ is Hamiltonian, then $G$ contains $W_6$ as a minor. Next, we introduce an important sufficient condition for a graph to be Hamiltonian. This condition allows us to disregard the specific connections among the six vertices of $G\setminus v$. As long as a certain degree sequence condition is satisfied, we can assert that the graph is Hamiltonian. For each integer $n\geq 1$, the nondecreasing sequence $d_1,d_2,\ldots,d_n$ is a \emph{degree sequence} of a graph $G$ if the vertices of $G$ can be labeled as $v_1,v_2,\ldots,v_n$ such that the degree $d(v_i)=d_i$ for each $i$.

	\begin{thm}[\cite{chvatal1972hamilton}]\label{thm 4.1}
	Suppose that the degree sequence of $G$ is $d_1\leq d_2\leq\ldots\leq d_n$, where $n\geq 3$. If for each $i<n/2$, either $d_i\geq i+1$ or $d_{n-i}\geq n-i$, then $G$ is Hamiltonian.
	\end{thm}

	By Theorem~\ref{thm 4.1}, we consider the degree sequence $d_1,d_2,\ldots,d_n$, where $n=6$.

	\begin{lem}\label{lem 4.2}
	Suppose that the degree sequence of $G$ is $d_1\leq d_2\leq\ldots\leq d_6$. If $d_1\geq2$ and $d_2\geq3$, then $G$ is Hamiltonian.
	\end{lem}

	\begin{lem}\label{lem 4.3}
	Suppose that the degree sequence of $G$ is $d_1\leq d_2\leq\ldots\leq d_6$. If $d_1=d_2=2$ and $d_3\geq 3$, then $G$ is Hamiltonian, except in the following two cases: the two degree-2 vertices are nonadjacent and have the same pair of neighbors; or $G$ is isomorphic to the graph $J$.
	\end{lem}

	\begin{figure}[H]
	\input{J.tpx}
	\end{figure}

	\begin{proof}
	Assume that $V(G)=\{x,y,v_1,v_2,v_3,v_4\}$ and $d(x)=d(y)=2$. We distinguish two cases according to whether $x$ and $y$ are adjacent.
 
 	\noindent\textbf{Case 1.} $x$ and $y$ are nonadjacent.
	First, let $N_G(x)\cap N_G(y)=\varnothing$. Assume $N_G(x)=\{v_1,v_2\}$ and $N_G(y)=\{v_3,v_4\}$. If $d(v_1)=3$ and $N_G(v_1)\setminus\{x\}=\{v_2,v_3\}$, then $v_2\in N_G(v_4)$ implies that $G$ is Hamiltonian. The same conclusion holds when $d(v_1)=3$ and $N_G(v_1)\setminus\{x\}=\{v_3,v_4\}$. If $d(v_1)=4$ and $N_G(v_1)\setminus\{x\}=\{v_2,v_3,v_4\}$, then at least one of $v_3$ and $v_4$ is adjacent to $v_2$. Thus, $G$ is Hamiltonian. Next, let $N_G(x)\cap N_G(y)= \{v_3\}$. By symmetry, assume that $v_4\notin N_G(x)\cup N_G(y)$. Then $N_G(v_4)=\{v_1,v_2,v_3\}$, which implies that $G$ is Hamiltonian. Finally, let $N_G(x)\cap N_G(y)=\{v_1,v_2\}$. Since $x$ and $y$ are nonadjacent, and $v_3$ and $v_4$ can only connect through $v_1$ and $v_2$, it is impossible to construct a closed cycle that includes all vertices. Thus, $G$ cannot be Hamiltonian.

	 \noindent\textbf{Case 2.} $x$ and $y$ are adjacent.
 	 Assume $(N_G(x)\cup N_G(y))\setminus\{x,y\}=\{v_1,v_2\}$. Then $N_G(v_3)=\{v_1,v_2,v_4\}$ and $N_G(v_4)=\{v_1,v_2,v_3\}$. Thus, $G$ is Hamiltonian. Otherwise, $(N_G(x)\cup N_G(y))\setminus\{x,y\}=\{v_1\}$. To ensure that $v_2$, $v_3$, and $v_4$ have degree at least 3, $v_1$ must be adjacent to all three vertices $v_2$, $v_3$, and $v_4$. This yields the unique graph $J$, which is not Hamiltonian. 
 	\end{proof}

	The following is a simple but useful observation.

	\begin{lem}\label{lem 4.4}
	Let $G$ be a $4$-connected graph on seven vertices. Then $G$ is $W_6$-minor-free if and only if $G$ contains no vertex of degree 6.
	\end{lem}

	\begin{proof}
	The sufficiency is immediate, since any graph on seven vertices containing $W_6$ as a minor must contain a vertex of degree 6.
		
	Conversely, suppose that $G$ contains a vertex $v$ of degree 6. If $v$ is removed, then the minimum degree of $G\setminus v$ is at least 3. By Lemma~\ref{lem 4.2}, $G\setminus v$ is Hamiltonian, which implies that $W_6\preceq G$.
	\end{proof}

	In the next few lemmas we determine all extensions of $C^2_5$ and $C^2_6$.
 
	\begin{lem}[\cite{ding2016graphs}]\label{lem4.5}
	The only splits of $C^2_5$ are $K_6$, $K_6\setminus e$, and $DW^+_4$.
	\end{lem}

	\begin{lem}\label{lem 4.6}
	Every split of $K_6$ contains a $W_6$-minor.
	\end{lem}

	\begin{proof}
	To establish this result, we may assume that both newly created vertices have degree 4. By symmetry, there is only one such split of $K_6$, denoted by $G$. Suppose that the two new vertices, $x$ and $y$, both have degree 4. There are five remaining vertices. Therefore, there exists a vertex $z$ that is adjacent to both $x$ and $y$. Thus, $z$ has degree 6. By Lemma~\ref{lem 4.4}, $W_6\preceq G$.
	\end{proof}

	\begin{lem} 
	The only $W_6$-minor-free splits of $K_6\setminus e$ are $K^{3,1}_{4,3}$ and $K^{4,1}_{4,3}$.
	\end{lem}

	\begin{proof}
	We first claim that splitting a degree-4 vertex of  $K_6\setminus e$  must result in a $W_6$-minor. The detailed process is omitted here as it is essentially the same as the proof of Lemma~\ref{lem 4.6}.
	\par
	Next, we consider the case of splitting a degree-5 vertex in $K_6\setminus e$. Suppose that both resulting vertices $x_1$ and $x_2$ have degree 4. Up to symmetry, there are precisely three distinct splits labeled $G_1$, $G_2$, and $G_3$, which are illustrated in Figure~\ref{fig5}. Both graphs $G_1$ and $G_2$ have vertices with degree 6, so both of them contain a $W_6$-minor by Lemma~\ref{lem 4.4}.

	\begin{figure}[H]
	\input{k6e.tpx}
	\label{fig5}
	\end{figure}

	Suppose at least one of the two new vertices has degree greater than 4. In this case, the resulting graph, denoted by $G$, is obtained from $G_1,G_2$, or $G_3$ by adding edges. If $G$ contains $G_1$ or $G_2$, then $G$ necessarily contains a $W_6$-minor. We consider only the case of $G_3$. Note that in the graph $G_3\setminus \{x_1,x_2\}$, there is exactly one vertex $y$ of degree 4, while all other vertices have degree 5. To avoid creating a vertex of degree 6, the new edge can only be incident with $y$, which is illustrated in Figure~\ref{fig6}. Thus, the only $W_6$-minor-free splits of $K_6\setminus e$ are $K^{3,1}_{4,3}$ and $K^{4,1}_{4,3}$.
	\end{proof}

	\begin{figure}[H]
	\input{k6ee.tpx}
	\label{fig6}
	\end{figure}
	
	\begin{lem}\label{lem 4.8}
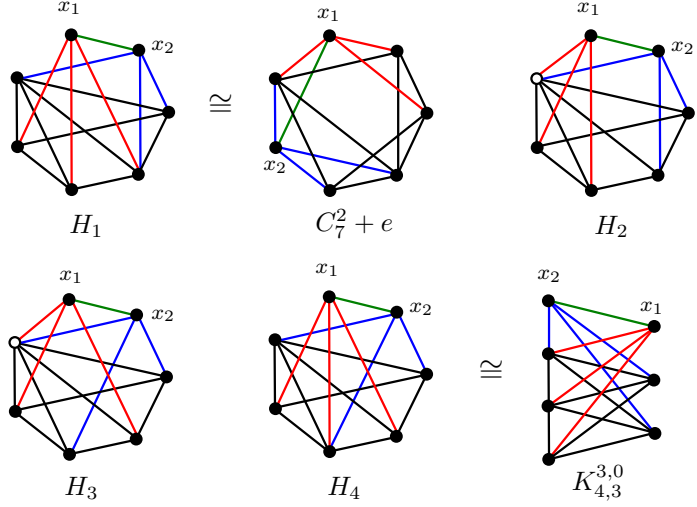
	The only $W_6$-minor-free splits of $DW^+_4$ are $C^2_7+e$, $K^{3,0}_{4,3}$, $K^{3,1}_{4,3}$, $K^{4,0}_{4,3}$, $\Gamma_1$, and $K^{4,1}_{4,3}$.
	\end{lem}

	\begin{proof}
		
	\noindent\textbf{Case 1.} Splitting a degree-4 vertex in $DW^+_4$.
	Assume that both resulting vertices $x_1$ and $x_2$ have degree 4. Up to symmetry, there are precisely three distinct splits labeled  $G_1$, $G_2$, and $G_3$, which are illustrated in Figure~\ref{fig7}. Both $G_1$ and $G_2$ have vertices with degree 6, so both of them contain a $W_6$-minor by Lemma~\ref{lem 4.4}.

 	\begin{figure}[H]
 	\input{dw4.tpx}
 	\label{fig7}
	\end{figure}
 
	Suppose at least one of the two new vertices has degree greater than 4. The resulting graph, denoted by $G$, is obtained from $G_1, G_2$, or $G_3$ by adding edges. We consider only the case of $G_3$. In the graph $G_3$, adding an edge results in a vertex of degree 6. Therefore, among the splits of degree-4 vertices of $DW^+_4$, the only $W_6$-minor-free split is $G_3$, which can also be denoted by $\Gamma_1$.
	
	\noindent\textbf{Case 2.} Splitting a degree-5 vertex in $DW^+_4$.
	Assume that both resulting vertices $x_1$ and $x_2$ have degree 4. Up to symmetry, there are precisely four distinct splits labeled $H_1$, $H_2$, $H_3$, and $H_4$, which are illustrated in Figure~\ref{fig8}. Both graphs $H_2$ and $H_3$ have vertices with degree 6 and therefore contain a $W_6$-minor. The other two graphs are isomorphic to $C^2_7+e$ and $K^{3,0}_{4,3}$, respectively, and hence are $W_6$-minor-free.
	
	\begin{figure}[H]
		\input{dw4e.tpx}
		\label{fig8}
	\end{figure} 
	
	Suppose at least one of the two new vertices has degree greater than 4. The resulting graph, denoted by $H$, is obtained from $H_1, H_2, H_3$, or $H_4$ by adding edges. We consider only the cases of $H_1$ and $H_4$. Up to symmetry, $H$ is isomorphic to $K^{3,1}_{4,3}, K^{4,0}_{4,3}, \Gamma_1$, or $K^{4,1}_{4,3}$.
	Therefore, among the splits of degree-5 vertices of $DW^+_4$, the only $W_6$-minor-free graphs are $C^2_7+e, K^{3,0}_{4,3}, K^{3,1}_{4,3}, K^{4,0}_{4,3}, \Gamma_1$, and $K^{4,1}_{4,3}$.
	\end{proof}

	\begin{lem} 
	The only planar $W_6$-minor-free split of $C^2_6$ is $DW_5$.
	\end{lem}

	\begin{proof}
	Splitting a vertex of $C^2_6$ results in two new vertices whose degree sum is 8, 9, or 10. The case in which the degree sum is 8 corresponds to the graphs $DW_5$ and $C^2_7+e$, while the cases in which the degree sum is 9 and 10 correspond to $\Gamma_1$ and $K^{4,1}_{4,3}$, respectively.
	
	By Lemma~\ref{lem 4.4}, $DW_5$ is $W_6$-minor-free. Moreover, the graphs $C^2_7+e$, $\Gamma_1$, and $K^{4,1}_{4,3}$ are nonplanar and were discussed in Lemma~\ref{lem 4.8}.
	\end{proof}

	The $W_6$-minor-free splits of $C^2_5$ and $C^2_6$ on seven vertices are $\mathcal{G}=\{C^2_7+e$, $DW_5$, $K^{3,0}_{4,3}$, $K^{3,1}_{4,3}$, $K^{4,0}_{4,3}$, $\Gamma_1$, $K^{4,1}_{4,3}\}$. Evidently, the graphs in $\mathcal{G}$ contain no vertices of degree 6 and have at least two vertices of degree 5. Let $\mathcal{H}$ be the set of graphs obtained by splitting graphs in $\mathcal{G}$.

	\begin{lem} \label{lem4.10}
	Each graph in $\mathcal{H}$ contains $W_6$ as a minor.
	\end{lem}

	\begin{proof}
	Suppose that the two new vertices have degree 4, and denote the resulting split by $H$. It is straightforward to verify that all vertex degrees in $H$ lie between 4 and 6, with at least one vertex exceeding 4.
	
	Let $x$ be a vertex such that $d(x)>4$. Assume $d(x)=5$ and $N_H(x)=\{v_1,v_2,v_3,v_4,y\}$. Then $v_5,v_6\notin N_H(x)$. Regardless of whether $v_5$ and $v_6$ are adjacent, both $v_5$ and $v_6$ are adjacent to at least three vertices in $N_H(x)$. Assume $y\in N_H(v_5)\cap N_H(v_6)$ and $N_H(x)\cup N_H(y)=\{v_1,v_2,v_3,v_4,v_5,v_6\}$. Then there exist two adjacent vertices $x$ and $y$ in $H$. Contracting the edge $xy$ results in a new vertex that is adjacent to all six remaining vertices $\{v_1,v_2,v_3,v_4,v_5,v_6\}$. The same argument applies when $d(x)>5$.
	
	Here we only consider the case where $d(v_i)=4$ for each $1\leq i\leq 6$, since any other situation would necessarily include this as a minor. If the combined degree of $x$ and $y$ is odd, one vertex among $\{v_1,v_2,v_3,v_4,v_5,v_6\}$ has degree 5, while the other five vertices have degree 4.
	
	\noindent\textbf{Case 1.} $d(x)=4$, $d(y)=4$ or $5$.
	
	 Contracting the edge $xy$ results in a new vertex of degree 6. The degree sequence of the remaining six vertices satisfies the condition of Lemma~\ref{lem 4.2}. Thus, $W_6\preceq H$.
	
	\noindent\textbf{Case 2.} $d(x)=4$, $d(y)=6$.
	
	Let $N_H(x)\setminus\{y\}=\{v_1,v_2,v_3\}$ and $N_H(y)\setminus\{x\}=\{v_2,v_3,v_4,v_5,v_6\}$. Contracting the edge $xy$ results in a new vertex of degree 6. The degree sequence of the remaining six vertices is $(2,2,3,3,3,3)$. By Lemma~\ref{lem 4.3}, we may assume that either $v_2$ and $v_3$ are nonadjacent and have the same pair of neighbors, or $H\setminus\{x,y\}$ is the graph $J$.
	
	We consider the latter situation first. Assume that, in the graph $J$, we have $d(v_1)=5$ and $d(v_5)=2$. Contracting the edge $v_1v_5$ of $H$ results in a new vertex of degree 6. The degree sequence of the remaining six vertices is at least $(2,3,3,3,3,3)$. Then $W_6\preceq H$ by Lemma~\ref{lem 4.2}. In the remainder of the proof, we omit the latter exceptional case in Lemma~\ref{lem 4.3}, as the above proof relies on the condition $d(x)\geq 4$ and $d(y)\geq 4$.
	
	Next, we consider the first exceptional case in Lemma~\ref{lem 4.3}. Note that $v_4$, $v_5$, and $v_6$ are structurally equivalent, meaning they have identical adjacency relationships with the remaining vertices. We assume that $v_2$ and $v_3$ are nonadjacent and share the same neighbors in $\{v_1,v_4\}$. Thus, neither $v_2$ nor $v_3$ belongs to $N_H(v_5)\cup N_H(v_6)$. Notice that $N_H(v_5)\setminus\{y\}=\{v_1,v_4,v_6\}$. Contracting the edge $yv_5$ results in a new vertex of degree 6. The degree sequence of the remaining six vertices becomes $(2,3,3,3,3,4)$. Similarly, we assume that $v_2$ and $v_3$ are nonadjacent and share the same neighbors in $\{v_4,v_5\}$. Thus, $N_H(v_1)\setminus\{y\}=\{v_4,v_5,v_6\}$, contracting the edge $xv_1$ results in a new vertex of degree 6. The degree sequence of the remaining six vertices becomes $(3,3,3,4,4,5)$. By Lemma~\ref{lem 4.2}, $W_6\preceq H$.
	
	\noindent\textbf{Case 3.} $d(x)=5$, $d(y)=5$.
	
	Let $N_H(x)\setminus\{y\}=\{v_1,v_2,v_3,v_4\}$ and $N_H(y)\setminus\{x\}=\{v_3,v_4,v_5,v_6\}$. Similar to Case 2, contracting the edge $yv_5$ or $yv_6$ results in a vertex of degree 6. The degree sequence of the remaining six vertices satisfies the condition of Lemma~\ref{lem 4.2}. Thus, $W_6\preceq H$.
	
	\noindent\textbf{Case 4.} $d(x)=5$, $d(y)=6$.
	
	Let $N_H(x)\setminus\{y\}=\{v_1,v_2,v_3,v_4\}$ and $N_H(y)\setminus\{x\}=\{v_2,v_3,v_4,v_5,v_6\}$. Note that $v_2$, $v_3$, and $v_4$ are structurally equivalent. Additionally, $v_1$ must be adjacent to at least one of these vertices. Let $v_1$ be adjacent to $v_2$. 
	
	Assume that $d(v_2)=4$. Contracting the edge $yv_2$ results in a new vertex of degree 6. The degree sequence of the remaining six vertices becomes $(2,3,3,3,3,3)$. There is exactly one vertex of degree 5 among $\{v_1,v_3,v_4,v_5,v_6\}$. The degree sequence is at least either $(2,3,3,3,3,4)$ or $(3,3,3,3,3,3)$.
	
 	Assume that $d(v_2)=5$. If $v_1$ is adjacent to $v_3$ or $v_4$, the reasoning is similar to the previous case where $d(v_2)=4$. Then $N_H(v_1)\setminus\{x\}=\{v_2,v_5,v_6\}$. Contracting the edge $xv_1$ results in a new vertex of degree 6. The degree sequence of the remaining six vertices becomes $(3,3,3,3,3,5)$. By Lemma~\ref{lem 4.2}, $W_6\preceq H$.
	
	\noindent\textbf{Case 5.} $d(x)=6$, $d(y)=6$.
	
	Let $N_H(x)\setminus\{y\}=\{v_1,v_2,v_3,v_4,v_5\}$ and $N_H(y)\setminus\{x\}=\{v_2,v_3,v_4,v_5,v_6\}$. Note that $v_2$, $v_3$, $v_4$, and $v_5$ are structurally equivalent, and $v_1$ must be adjacent to at least two of these vertices. Assume $v_2,v_3\in N_H(v_1)$. Contracting the edge $yv_2$ results in a new vertex of degree 6. The degree sequence of the remaining six vertices becomes $(2,3,3,3,3,4)$. By Lemma~\ref{lem 4.2}, $W_6\preceq H$.
	
	Now, suppose at least one of the two new vertices has degree greater than 4. The resulting graph, denoted by $H'$, is obtained from $H$ by adding edges. Since $W_6\preceq H$, it follows that $W_6\preceq H'$.
	\end{proof}

	\begin{proof}[\textbf{Proof of Theorem~\ref{thm1.1}}]
	
	The result follows from Theorem~\ref{thm2.1} and Lemmas~\ref{Lemma3.3}--\ref{Lemma3.4} and~\ref{lem4.5}--\ref{lem4.10}.	
	\end{proof}

	\section*{Acknowledgments}
	This work was supported by the National Natural Science Foundation of China (No.~12371356).

%\bibliography{w6bil}

\begin{thebibliography}{99}
	
	\bibitem{chvatal1972hamilton}
	V.~Chv{\'a}tal.
	\newblock On {Hamilton}'s ideals.
	\newblock {\em Journal of Combinatorial Theory, Series B}, 12(2):163--168, 1972.
	
	\bibitem{ding2013characterization}
	G.~Ding.
	\newblock A characterization of graphs with no octahedron minor.
	\newblock {\em Journal of Graph Theory}, 74(2):143--162, 2013.
	
	\bibitem{ding2016graphs}
	G.~Ding, C.~Lewchalermvongs, and J.~Maharry.
	\newblock Graphs with no \(\overline{P}_7\)-minor.
	\newblock {\em The Electronic Journal of Combinatorics}, 23(2):2--16, 2016.
	
	\bibitem{dirac1952property}
	G.A. Dirac.
	\newblock A property of 4-chromatic graphs and some remarks on critical graphs.
	\newblock {\em Journal of the London Mathematical Society}, 1(1):85--92, 1952.
	
	\bibitem{gaslowitz2018characterizations}
	J.Z. Gaslowitz.
	\newblock Characterizations of graphs without certain small minors.
	\newblock {\em Vanderbilt University}, 2018.
	
	\bibitem{gubser1993planar}
	B.S. Gubser.
	\newblock Planar graphs with no 6-wheel minor.
	\newblock {\em Discrete Mathematics}, 120(1-3):59--73, 1993.
	
	\bibitem{maharry1999excluded}
	J.~Maharry.
	\newblock An excluded minor theorem for the octahedron.
	\newblock {\em Journal of Graph Theory}, 31(2):95--100, 1999.
	
	\bibitem{maharry2000characterization}
	J.~Maharry.
	\newblock A characterization of graphs with no cube minor.
	\newblock {\em Journal of Combinatorial Theory, Series B}, 80(2):179--201, 2000.
	
	\bibitem{maharry2016structure}
	J.~Maharry and N.~Robertson.
	\newblock The structure of graphs not topologically containing the {W}agner graph.
	\newblock {\em Journal of Combinatorial Theory, Series B}, 121:398--420, 2016.
	
	\bibitem{martinov1982uncontractable}
	N.~Martinov.
	\newblock Uncontractable 4-connected graphs.
	\newblock {\em Journal of Graph Theory}, 6(3):343--344, 1982.
	
	\bibitem{oxley1989regular}
	J.G. Oxley.
	\newblock The regular matroids with no 5-wheel minor.
	\newblock {\em Journal of Combinatorial Theory, Series B}, 46(3):292--305, 1989.
	
	\bibitem{qin2019chain}
	C.~Qin and G.~Ding.
	\newblock A chain theorem for 4-connected graphs.
	\newblock {\em Journal of Combinatorial Theory, Series B}, 134:341--349, 2019.
	
	\bibitem{tutte1961theory}
	W.T. Tutte.
	\newblock A theory of 3-connected graphs.
	\newblock {\em Nederl. Akad. Wetensch. Proc. Ser. A}, 64:441--455, 1961.
	
	\bibitem{tutte1966algebraic}
	W.T. Tutte.
	\newblock On the algebraic theory of graph colorings.
	\newblock {\em Journal of Combinatorial Theory, Series B}, 1(1):15--50, 1966.
	
	\bibitem{wormald1979classifying}
	N.C. Wormald.
	\newblock Classifying k-connected cubic graphs.
	\newblock {\em Lecture Notes in Mathematics}, 748:199--206, 1979.
	
\end{thebibliography}

\end{document}